\documentclass{article}

\usepackage{amsmath,amssymb,amsfonts,amsthm}
\usepackage[all]{xy}
\usepackage{color}
\usepackage{dsfont}

\usepackage{hyperref}%To get clickable links to papers

\usepackage{multirow}
\usepackage{rotating}
\usepackage{fullpage}

\renewcommand{\(}{\begin{equation}}
\renewcommand{\)}{\end{equation}}
\newcommand{\bea}{\begin{eqnarray}}
\newcommand{\eea}{\end{eqnarray}}

\newcommand{\C}{{\mathbb C}}
\newcommand{\Z}{{\mathbb Z}}

\newtheorem{theorem}{Theorem}[subsection]
\newtheorem{lemma}[theorem]{Lemma}

\newtheorem{proposition}[theorem]{Proposition}

\newtheorem{definition}[theorem]{Definition}

\newtheorem{remark}[theorem]{Remark}

\begin{document}

\begin{flushright}
   {\sf ZMP-HH/14-2}\\
   {\sf Hamburger$\;$Beitr\"age$\;$zur$\;$Mathematik$\;$Nr.$\;$499}\\[2mm]
   January 2014
\end{flushright}
\vskip 3.5em
                                   
\begin{center}
\begin{tabular}c \Large\bf 
   A Serre-Swan theorem for gerbe modules \\[2mm] \Large\bf
   on \'etale Lie groupoids
\end{tabular}\vskip 2.5em
                                     
  ~Christoph Schweigert\,$^{\,a}$,~
  ~Christopher Tropp\,$^{\,b}$,~
  ~Alessandro Valentino\,$^{\,a}$

\vskip 9mm

  \it$^a$
  Fachbereich Mathematik, \ Universit\"at Hamburg\\
  Bereich Algebra und Zahlentheorie\\
  Bundesstra\ss e 55, \ D\,--\,20\,146\, Hamburg\\[4pt]
  \it$^b$
  Mathematisches Institut, \ WWU M\"unster\\
  Einsteinstr. 62, \ D\,--\, 48149 M\"unster\\[4pt]

\end{center}
                     
\vskip 5.3em

\noindent{\sc Abstract}\\[3pt]
Given a bundle gerbe on a compact smooth manifold or, more generally, 
on a compact \'etale Lie groupoid $M$, we show that the corresponding 
category of gerbe modules, if it is non-trivial, is equivalent
to the category of finitely generated projective modules over an
Azumaya algebra on $M$. This result can be seen as an equivariant
Serre-Swan theorem for twisted vector bundles.

\section{Introduction}
The celebrated Serre-Swan theorem relates the category of vector
bundles over a compact smooth manifold $M$ to the category of finite rank
projective modules over the algebra of smooth functions $C^\infty(M,\C)$
of $M$ (see \cite{gbv, mory} for the Serre-Swan theorem in the smooth category).
It relates geometric and algebraic notions and is, in particular,
the starting point for the definition of vector bundles in 
non-commutative geometry.

A bundle gerbe on $M$ can be seen as a geometric realization of
its Dixmier-Douady class, which is a class in $H^3(M;\Z)$. 
To such a geometric realization, a twisted
$K$-theory group can be associated. Gerbe modules have been introduced to
obtain a geometric description of twisted K-theory \cite{BCMMS}.
A bundle gerbe (with connection) describes
a string background with non-trivial $B$-field; gerbe modules
(with connection) arise also as Chan-Paton bundles on the worldvolume $M$ of D-branes in such backgrounds
\cite{gaw}. It is an old idea that, in the presence of a non-trivial
$B$-field, the worldvolume of a D-brane should become non-commutative
in some appropriate sense.
This has lead to the idea \cite{kapustin} that an Azumaya algebra over 
$M$ then plays the role of the algebra of functions on $M$ and that
a version of the Serre-Swan theorem should relate gerbe modules and finitely
generated projective modules over this Azumaya algebra.
This idea has been made mathematically rigourous in \cite{karoubi},
using the language of twisted vector bundles which requires
using a suitable open cover of $M$ and working with
locally defined quantities.

In this note we first derive this result, but with rather different
techniques, using descent theory. Our key insight can be described
as follows (cf.\ Lemma 3.1.1): a gerbe module on $M$ is a vector bundle $E$ on 
the space $Y$ of a fibration $Y\to M$, together with additional
data. We show that, given two gerbe modules $E,E'$ over the same
gerbe, the homomorphism bundle $\mathrm{Hom}(E,E')$ is not only a vector
bundle over $Y$, but comes with enough data to turn it into
an object in the descent category of vector bundles
$Desc(Y\to M)$. We then identify this category with the category
of vector bundles on $M$ and use the global section functor
in the spirit of the Serre-Swan theorem.
This yields an Azumaya algebra on $M$;
our construction is natural and thus yields for any gerbe module a
module over this Azumaya algebra.

The stack of vector bundles naturally extends from smooth manifolds
to Lie groupoids. Since our techniques are based on general descent
techniques, they can be transferred to 
\'etale groupoids, so that we finally obtain theorem \ref{4.3.5},
a Serre-Swan theorem for gerbe modules on \'etale Lie groupoids.
Action groupoids provide examples of \'etale Lie groupoids.
Our results, apart from their intrinsic mathematical interest, 
therefore have applications to D-branes in string backgrounds with
group actions and to D-branes in orbifolds with $B$-fields.

Our note is organized as follows: section 2 contains some preliminaries;
in section 3, we use descent techniques to prove the Serre-Swan 
theorem \ref{main} for gerbe modules on smooth compact manifolds. 
In section 4, we generalize our results to \'etale Lie groupoids.

\section{Preliminaries}
\subsection{Gerbes and gerbe modules}\label{prelim}

In this section we will recall some background material on gerbes and gerbe modules, referring to \cite{murray, waldorf} for further details. We stress that all gerbes and gerbe modules are \emph{not} equipped with a connection.\\
In the following, given a fibration $Y\to{M}$, $Y_{M}^{[k]}$ will denote the $k$-th fibered product of $Y$ over $M$, and $\pi_{i}:Y_{M}^{[k]}\to Y_{M}^{[k-1]}$ the map given by omitting the $i$-th entry.
Similarly, $\pi_{i_{1}i_{2}\ldots{i_{m}}}:Y^{[k]}\to Y^{[k-m]}$ denotes 
the composition $\pi_{i_{1}}\circ\pi_{i_{2}}\circ\ldots\circ
\pi_{i_{m}}$. 

\begin{definition}
A bundle gerbe $\mathcal{G}$ over a manifold $M$ consists of a triple $(Y,L,\mu)$, where $\pi:Y\to{M}$ is a surjective submersion, $L$ a hermitian line bundle over $Y_{M}^{[2]}$, and 
\begin{displaymath}
\mu:\pi_{3}^{*}L\otimes\pi_{1}^{*}L\to\pi^{*}_{2}L
\end{displaymath}
is a bundle isomorphism over $Y_{M}^{[3]}$ satisfying the natural associativity condition over $Y_{M}^{[4]}$. 
\end{definition} 

Given a bundle gerbe $\mathcal{G}$, we can introduce the notion of $\mathcal{G}$-modules and their morphisms. 
\begin{definition} \label{gmodule}
Let $\mathcal{G}=(Y,L,\mu)$ be a bundle gerbe over $M$. A gerbe module $\mathcal{M}$ over $\mathcal{G}$ (or $\mathcal{G}$-module) is a pair $(E,\rho)$, where $E\to{Y}$ is a finite rank hermitian vector bundle, and 
\begin{displaymath}
\rho:L\otimes\pi_{1}^{*}E \to \pi_{2}^{*}E
\end{displaymath}
is an isomorphism of hermitian vector bundles on $Y^{[2]}$, 
satisfying a compatibility condition over $Y^{[3]}$, namely the bundle maps obtained from pullbacks of $\rho$ and $\mu$
\begin{displaymath}
\pi^{*}_{3}L\otimes\pi_{1}^{*}(L\otimes\pi_{1}^{*}E) \to \pi^{*}_{3}L\otimes\pi_{21}^{*}E \to \pi_{23}^{*}E
\end{displaymath} 
and
\begin{displaymath}
\pi^{*}_{3}L\otimes\pi_{1}^{*}(L\otimes\pi_{1}^{*}E) \to \pi^{*}_{2}L\otimes\pi_{12}^{*}E \to \pi_{23}^{*}E
\end{displaymath} 
coincide.
\end{definition}
Any bundle gerbe admits a trivial gerbe module, given by the pair $(\underline{0},id)$.
\begin{definition}
Let $\mathcal{G}=(Y,L,\mu)$ be a bundle gerbe, and $\mathcal{M}=(E,\rho)$, $\mathcal{N}=(E',\rho')$ be $\mathcal{G}$-modules. A $\mathcal{G}$-module morphism $\mathcal{M}\to\mathcal{N}$ is given by a vector bundle morphism $f:E\to{E'}$ such that the following diagram of morphisms of
hermitian vector bundles on $Y^{[2]}$
%%%
\begin{equation}\label{gmorf}
\xymatrix{L\otimes\pi_{1}^{*}E\ar[r]^{\rho}\ar[d]_{id\otimes\pi_{1}^{*}f} & \pi_{2}^{*}E\ar[d]^{\pi_{2}^{*}f}\\
L\otimes\pi_{1}^{*}E'\ar[r]^{\rho'} & \pi_{2}^{*}E'
}
%%%
\end{equation}
commutes. We will denote with ${\rm Hom}_{\mathcal{G}}(\mathcal{M},\mathcal{N})$ the space of $\mathcal{G}$-module morphisms between $\mathcal{M}$ and $\mathcal{N}$.
\end{definition}
$\mathcal{G}$-modules and their morphisms form a category $\mathcal{G}{\rm -mod}$. Moreover, the direct sum of vector bundles and bundle morphisms induces a direct sum on $\mathcal{G}{\rm -mod}$, with the trivial gerbe module $(\underline{0},{\rm id})$ as the neutral element.\\
It is not difficult to prove the following
\begin{lemma}Let $\mathcal{G}$ be a bundle gerbe over $M$. Then $\mathcal{G}{\rm -mod}$ is a $\mathbb{C}$-linear category.\end{lemma}
\subsection{Azumaya algebras over manifolds}
Recall that an \emph{Azumaya algebra} over a commutative (local) ring $R$ is an $R$-algebra $A$ such that:
\begin{enumerate}
\item $A$ is free and finite rank as an $R$-module
\item $A\otimes_{R}A^{opp}\simeq{\rm End}_{R}(A)$ via the assignment $a\otimes{b}\to a\cdot(\quad)\cdot{b}$.
\end{enumerate}  

\begin{lemma}
For a complex vector space $V$, the endomorphism algebra ${\rm End}(V)$ is an Azumaya algebra.
\end{lemma}    

A generalization of the notion of Azumaya algebra in the context of algebraic geometry is due to Grothendieck \cite{grothendieck}. In the following we give the relevant definition in the smooth setting.
\begin{definition}
An Azumaya algebra over a manifold $M$ is a complex algebra $\mathcal{A}$ which can be obtained as the algebra of sections of an algebra bundle on $M$ whose  fibers are Azumaya algebras.  
\end{definition} 
In the following, $\mathcal{A}$ will denote either the algebra of section or the actual algebra bundle, and we will pass between the two equivalent descriptions freely. Moreover, we will often omit to indicate the manifold $M$.\\
An example of an Azumaya algebra over a manifold $M$ is given by the sections of the endomorphism bundle ${\rm End}(E)$, where $E$ is a complex vector bundle. \\
Azumaya algebras over $M$ are equipped with a tensor product, given by the tensor product of algebra bundles. We will say that two Azumaya algebras $\mathcal{A}_{1}$ and $\mathcal{A}_{2}$ are equivalent if there exist vector bundles $E_{1}$ and $E_{2}$ such that
\begin{displaymath}
\mathcal{A}_{1}\otimes{\rm End}(E_{1}) \simeq \mathcal{A}_{2}\otimes{\rm End}(E_{2})
\end{displaymath}
%%%
The set of equivalence classes of Azumaya algebras over $M$ 
forms a group, the \emph{Brauer group} of $M$ with inverses 
represented by opposite algebras.
An important result is the following \cite{grothendieck}
%%%
\begin{theorem}
The Brauer group of a manifold $M$ is isomorphic to the torsion subgroup of $H^{3}(M;\mathbb{Z})$.
\end{theorem}
%%%
\section{Gerbe modules and descent data}
\subsection{Descent category}\label{descent}
Let $Y\xrightarrow{\pi}{M}$ be a surjective submersion, and let $S:Man^{op}\to{Cat}$ a pre-sheaf over the site of manifolds with values in categories. To these data we can assign the \emph{descent category} $Desc^{S}(Y\xrightarrow{\pi}{M})$ defined as follows \cite{brylinski, niksch,husemoller}:
\begin{enumerate}
\item an object is a pair $(E,\varphi)$, where $E$ is an object in $S(Y)$ and $\varphi$ is an isomorphism 
\begin{displaymath}
\varphi:\pi_{1}^{*}{E}\simeq\pi_{2}^{*}E
\end{displaymath}
over $Y_{M}^{[2]}$ satisfying the associativity condition $\pi_{2}^{*}(\varphi)=\pi_{3}^{*}(\varphi)\circ\pi_{1}^{*}(\varphi)$ over $Y_{M}^{[3]}$.
\item a morphism $(E,\varphi)\to(E',\varphi')$ is given by an element $f\in{\rm Hom}_{S(Y)}(E,E')$ for which the following diagram commutes
%%%
\begin{displaymath}
\xymatrix{\pi_{1}^{*}{E}\ar[r]^{\varphi}\ar[d]_{\pi^{*}_{1}(f)} & \pi_{2}^{*}E\ar[d]^{\pi^{*}_{2}(f)}\\
\pi_{1}^{*}{E'}\ar[r]^{\varphi'} & \pi_{2}^{*}E'
}
\end{displaymath}
%%%
\end{enumerate}
We have a canonical functor $\pi^{*}:S(M)\to{Desc^{S}(Y\xrightarrow{\pi}{M})}$, which assigns to $E\in{S(M)}$ the pair $(\pi^{*}E,{\rm id})$. Moreover, when $S$ is a stack for surjective submersions, the functor $\pi^{*}$ is an equivalence of categories.\\
We will be interested in the case where $S={\rm Vect}$, the stack of vector bundles. To simplify the notation, we will denote with $Desc(Y\xrightarrow{\pi}{M})$ the descent category associated to $Y\xrightarrow{\pi}{M}$ and ${\rm Vect}$. In this case, $\pi^{*}$ has a canonical inverse 
$D:Desc(Y\xrightarrow{\pi}{M})\to{\rm Vect}(M)$ (See \cite{brylinski}, Chapter 5).\\

Let $\mathcal{G}=(Y,L,\mu)$ be a bundle gerbe over $M$, and let $\mathcal{M}=(E,\rho)$ and $\mathcal{N}=(E',\rho')$ be $\mathcal{G}$-modules. Consider the homomorphism bundle ${\rm Hom}(E,E')\simeq E^{*}\otimes E'$ on $Y$. 
\footnote{Whenever we refer to homomorphism bundles, the symbol 
$\mathrm{Hom}$ is used without subscript; homomorphism sets in categories
in contrast always have a subscript.}
Consider the isomorphism 
\begin{equation}
 \varphi_{EE'}:\pi_{2}^{*}{\rm Hom}(E,E')\to\pi_{1}^{*}{\rm Hom}(E,E')
\end{equation}
%%%
of hermitian line bundles on $Y^{[2]}$ induced by 
\begin{equation}\label{desciso}
\pi_{2}^{*}E^{*}\otimes\pi_{2}^{*}E'\xrightarrow{(\rho)^{*}\otimes(\rho')^{-1}}\pi_{1}^{*}E^{*}\otimes(L^{*} \otimes L)\otimes\pi_{1}^{*}E'\xrightarrow{{\rm id}\otimes{c}\otimes{\rm id}}\pi_{1}^{*}E^{*}\otimes\pi_{1}^{*}E'
\end{equation}
%%%
where $c:L^{*}\otimes{L}\to\underline{\mathbb{C}}$ denotes the canonical isomorphism.\\
We have then the following central lemma
\begin{lemma}
\label{lemmaobj}
$({\rm Hom}(E,E'),\varphi_{EE'}^{-1})$ is an object in $Desc(Y\xrightarrow{\pi}{M})$.
\end{lemma}
%%%
%%%%
\begin{proof}
To simplify the notation, we set $\varphi=\varphi_{EE'}$. We have to prove that $\varphi^{-1}$ satisfies the associativity condition $\pi_{2}^{*}\varphi^{-1}=\pi_{3}^{*}\varphi^{-1}\circ\pi_{1}^{*}\varphi^{-1}$, or, equivalently, that $\varphi$ satisfies $\pi_{2}^{*}\varphi=\pi_{1}^{*}\varphi\circ\pi_{3}^{*}\varphi$. We will simplify the notation by using $E_{ij}$, $\varphi_{ij}$, etc. for $\pi^{*}_{ij}E$, $\pi_{ij}^{*}\varphi$, etc. .\\
First, recall that the isomorphism $\mu$ sits in the following commutative diagram
\begin{equation}\label{diagmu}
\xymatrixcolsep{5pc}\xymatrix{L^{*}_{1}\otimes L^{*}_{3}\otimes L_{3} \otimes L_{1} \ar[r]^{(\mu^{*})^{-1}\otimes\mu}\ar[d]_{c_{3}} & L^{*}_{2}\otimes L_{2}\ar[d]^{c_{2}}\\
L^{*}_{1} \otimes L_{1} \ar[r]_{c_{1}} & \underline{\mathbb{C}}
}
\end{equation}
where we have used the canonical pairing $c_{i}:L^{*}_{i}\otimes{L_{i}}\to\underline{\mathbb{C}}$. By using that the morphism $\rho$ and $\rho'$ are compatible with $\mu$, we see that the morphism $\varphi_{2}$ can be obtained as
%%%
\begin{equation}\label{varphi2}
\xymatrixcolsep{5pc}\xymatrix{
E_{23}^{*}\otimes{E_{23}'}\ar[r]^{\rho_{3}^{*}\otimes(\rho_{3}')^{-1}} & E_{13}^{*} \otimes L_{3}^{*} \otimes L_{3} \otimes E_{13}'\ar[r]^{\rho_{1}^{*}\otimes(\rho_{1}')^{-1}} &  E_{12}^{*} \otimes L_{1}^{*} \otimes L_{3}^{*} \otimes L_{3} \otimes L_{1} \otimes E_{12}' \ar[d]^{(\mu^{*})^{-1}\otimes\mu} \\
& &  E_{12}^{*} \otimes L_{2}^{*} \otimes L_{2} \otimes E_{12}'\ar[d]^{c_{2}}\\
& &  E_{12}^{*} \otimes E_{12}'
}
\end{equation}
%%%
Moreover, we have the following commutative diagram
%%%
\begin{equation}\label{diagc}
\xymatrixcolsep{5pc}\xymatrix{
E_{13}^{*} \otimes L_{3}^{*} \otimes L_{3} \otimes E_{13}' \ar[r]^{c_{3}} \ar[d]_{\rho_{1}^{*} \otimes (\rho_{1}')^{-1}} & E_{13}^{*} \otimes E'_{13} \ar[d]^{\rho_{1}^{*} \otimes (\rho_{1}')^{-1}}\\
E_{12}^{*} \otimes L_{1}^{*} \otimes L_{3}^{*} \otimes L_{3} \otimes L_{1} \otimes E_{12}' \ar[r]_{c_{3}} & E_{12}^{*} \otimes L_{1}^{*} \otimes L_{1} \otimes  E_{12}' 
}
\end{equation}
%%%
The proof is completed by the following commutative diagram
\begin{equation}
\xymatrixcolsep{5pc}\xymatrix{
E_{23}^{*}\otimes E_{23}'\ar[r]^{\varphi_{3}} \ar[d]^{\rho_{3}^{*}\otimes(\rho_{3}')^{-1}} & E_{13}^{*}\otimes E'_{13} \ar[r]^{\varphi_{1}} \ar[d]_{\rho_{1}^{*}\otimes(\rho_{1}')^{-1}}& E_{12}^{*}\otimes E_{12}'\\
E_{13}^{*}\otimes L_{3}^{*}\otimes L_{3} \otimes E'_{13} \ar[ru]^{c_{3}}\ar[d]^{\rho_{1}^{*}\otimes(\rho_{1}')^{-1}} & E_{12}^{*}\otimes L_{1}^{*}\otimes L_{1} \otimes E'_{12}\ar[ru]^{c_{1}} &\\
E_{13}^{*}\otimes L_{1}^{*} \otimes L_{3}^{*}\otimes L_{3} \otimes L_{1} \otimes E'_{13} \ar[rr]_{(\mu^{*})^{-1}\otimes\mu}\ar[ru]^{c_{3}}& &E_{13}^{*}\otimes L_{2}^{*} \otimes  L_{2} \otimes E'_{13}\ar[uu]^{c_{3}}
}
\end{equation}
Notice that the whole diagram is commutative, since every single subdiagram is commutative as a consequence of the definition of $\varphi_{3},\varphi_{1}$, diagram (\ref{diagmu}) and (\ref{diagc}). In particular, the outer square is commutative: by diagram (\ref{varphi2}), the morphism obtained
as the composition of the left, lower and right outer edge coincides 
with the morphism $\varphi_{2}$, hence we have that $\varphi_{2}=\varphi_{1}\circ\varphi_{3}$.\\
%%%
\end{proof}
%%%
\begin{lemma}
\label{lemmamorf}
Let $\mathcal{M}=(E,\rho)$, $\mathcal{N}=(E',\rho')$ and
$\mathcal{P}=(E'',\rho'')$ be $\mathcal{G}$-modules, and let $f\in{\rm
  Hom}_{\mathcal{G}}(\mathcal{M},\mathcal{N})$. Then $f$ induces a
morphism $({\rm Hom}(E,E'),\varphi_{EE'}^{-1})\to({\rm
  Hom}(E,E''),\varphi_{EE''}^{-1})$ in $Desc(Y\xrightarrow{\pi}{M})$.
\end{lemma}
\begin{proof}
Recall that a gerbe module morphism is an (hermitian) vector bundle morphism satisfying a compatibility condition with the morphism $\rho$, namely diagram (\ref{gmorf}). If we denote by abuse of notation the morphism of vector bundles
by $f$ as well, $f:E'\to E''$, it is then straightforward to show that the morphism
\begin{equation}
\begin{array}{rl}
{\rm Hom}(E,E')&\to {\rm Hom}(E,E'')\\
\beta &\to f\circ\beta
\end{array}
\end{equation}
satisfies the condition for a morphism in the descent category.
\end{proof}
Let $\mathcal{M}=(E,\rho)$ be a nontrivial $\mathcal{G}$-module. Then 
Lemma \ref{lemmaobj} and Lemma \ref{lemmamorf} 
guarantee that we have a functor 
\begin{equation}
\begin{array}{rl}
\tilde{\mathcal{F}}^{\mathcal{M}}:\mathcal{G}-{\rm mod}&\to Desc(Y\xrightarrow{\pi}{M})\\
(E',\rho')&\mapsto ({\rm Hom}(E,E'),\varphi_{EE'}^{-1})\\
f&\mapsto(\beta\to (f\circ\beta))
\end{array}
\end{equation}
Moreover, we have
\begin{proposition}
The functor $\tilde{\mathcal{F}}^{\mathcal{M}}$ is a faithful $\mathbb{C}$-linear functor.
\end{proposition}
\begin{proof}
Since the functor ${\rm Hom}(E,-)$ is $\mathbb{C}$-linear and faithful, the functor $\tilde{\mathcal{F}}^{\mathcal{M}}$ inherits these properties. 
\end{proof}
\begin{proposition}\label{azumend}
To any nontrivial $\mathcal{G}$-module $\mathcal{M}=(E,\rho)$ we can canonically associate an Azumaya algebra $\mathcal{A}^{\mathcal{M}}$ over $M$.
\end{proposition}
\begin{proof}
Define 
\begin{equation}
\mathcal{A}^{\mathcal{M}}:=\Gamma(M,D(\tilde{\mathcal{F}}^{\mathcal{M}}(\mathcal{M})))
\end{equation}
where $D$ is the canonical inverse $D:Desc(Y\xrightarrow{\pi}{M})
\to{\rm Vect}(M)$ and $\Gamma$ is the global section functor.\\
Notice that  ${\rm End}(E):={\rm Hom}(E,E)$ is an algebra bundle over $Y$ and $\varphi_{EE}^{-1}$ is an algebra bundle isomorphism. Since $D$ is a tensor functor, then the bundle $D(\tilde{\mathcal{F}}^{\mathcal{M}}(\mathcal{M}))$ is an algebra bundle over $M$ such that $\pi^{*}(D(\tilde{\mathcal{F}}^{\mathcal{M}}(\mathcal{M})))\simeq{\rm End}(E)$. Moreover, since $\Gamma(Y,{\rm End}(E))$ is an Azumaya algebra over $Y$, the same is true for $\mathcal{A}^{\mathcal{M}}$. 
\end{proof}
\subsection{Equivalence of categories}\label{equivcat}
The Serre-Swan theorem states that the section functor $\Gamma$ induces an equivalence between the category of vector bundles over a compact manifold $M$ and the category of finitely generated projective $C^{\infty}(M;\C)$-modules.\\
Let $\mathcal{M}=(E,\rho)$ be a nontrivial $\mathcal{G}$-module, and consider the category ${\rm pfmod-}\mathcal{A}^{\mathcal{M}}$ of  projective 
and finitely generated right modules over the Azumaya algebra $\mathcal{A}^{\mathcal{M}}$ over $M$.\\
Let $\mathcal{N}=(E',\rho')$ be an arbitrary $\mathcal{G}$-module, and define
\begin{equation}
\mathcal{A}^{\mathcal{M}\mathcal{N}}:=\Gamma(M,D(\tilde{\mathcal{F}}^{\mathcal{M}}(\mathcal{N})))
\end{equation}
\begin{lemma}
$\mathcal{A}^{\mathcal{M}\mathcal{N}}$ is a projective and finitely generated right module over $\mathcal{A}^{\mathcal{M}}$.
\end{lemma}
\begin{proof}
First, notice that ${\rm Hom}(E,E')$ is a finitely generated right module over ${\rm End}(E)$. It is moreover fibrewise projective, since the fibers are finite-dimensional modules over the endomorphism algebra of a finite dimensional vector space, i.e. a full matrix algebra. Taking into account the isomorphisms $\varphi_{EE'}$ and $\varphi_{EE}$, we have that $\tilde{\mathcal{F}}^{\mathcal{M}}(\mathcal{N})$ is right $\tilde{\mathcal{F}}^{\mathcal{M}}(\mathcal{M})$-module in $Desc(Y\xrightarrow{\pi}{M})$. Finally, use that both $D$ and $\Gamma$ are tensor functors, and that they preserve projectivity.\\    
\end{proof}
\begin{lemma}
A morphism $\mathcal{N}\to\mathcal{P}$ of $\mathcal{G}$-modules induces a morphism $\mathcal{A}^{\mathcal{M}\mathcal{N}}\to\mathcal{A}^{\mathcal{M}\mathcal{P}}$ of $\mathcal{A}^{\mathcal{M}}$-modules. 
\end{lemma}
\begin{proof}
Use that any bundle morphism $E'\to{E''}$ induces a morphism ${\rm Hom}(E,E')\to{\rm Hom}(E,E'')$ which is also an ${\rm End}(E)$-module morphism.
\end{proof}
The results above allow us to define a functor 
\begin{equation}
\begin{array}{rl}
\mathcal{F}^{\mathcal{M}}:\mathcal{G}{\rm -mod}&\to {\rm pfmod-}\mathcal{A}^{\mathcal{M}}\\
\mathcal{N}&\mapsto \mathcal{A}^{\mathcal{M}\mathcal{N}}\\
\mathcal{N}\to\mathcal{P}&\mapsto \mathcal{A}^{\mathcal{M}\mathcal{N}} \to \mathcal{A}^{\mathcal{M}\mathcal{P}}
\end{array}
\end{equation} 
\begin{theorem}\label{main}
The functor $\mathcal{F}^{\mathcal{M}}$ is a fully faithful and essentially surjective $\mathbb{C}$-linear functor. 
\end{theorem}
\begin{proof}
The faithfulness and $\mathbb{C}$-linearity is guaranteed by the fact that the functors $\mathcal{F}^{\mathcal{M}}$, $D$, and $\Gamma$ are faithful and $\mathbb{C}$-linear. We have then to prove fullness and essential surjectivety.
\begin{enumerate}
\item Fullness: Let $f\in{\rm Hom}_{\mathcal{A}^{\mathcal{M}}}(\mathcal{A}^{\mathcal{M}\mathcal{N}},\mathcal{A}^{\mathcal{M}\mathcal{P}})$. Since both $\Gamma$ and $D$ are full tensor functors, to $f$ there corresponds a morphism ${\rm Hom}(E,E')\to{\rm Hom}(E,E'')$ which is also an ${\rm End}(E)$-module morphism. We have the following canonical isomorphism
\begin{equation}
{\rm Hom}({\rm Hom}(E,E'),{\rm Hom}(E,E''))\simeq {\rm End}(E)\otimes {\rm Hom}(E',E'')\,\,;
\end{equation}
The space of bundle homomorphisms ${\rm Hom}_{\rm Vect}(E,E')$
is then obtained from the space of bundle morphisms
from the trivial bundle $\underline{\mathbb{C}}$ to the homomorphism
bundle, ${\rm Hom}_{\rm Vect}(\underline{\mathbb{C}},{\rm Hom}(E,E'))\cong{\rm Hom}_{\rm Vect}(E,E')$.
As consequence, ${\rm Hom}_{\rm Vect}({\rm Hom}(E,E'),{\rm Hom}(E,E''))\cong
{\rm Hom}_{\rm Vect}(\underline{\mathbb{C}}, {\rm End}(E)\otimes{\rm Hom}(E',E'')$.
Under this isomorphism, the action of the algebra bundle
${\rm End}(E)$ is multiplication on the first factor.
Hence those bundle morphisms in ${\rm Hom}_{\rm Vect}({\rm Hom}(E,E'),{\rm Hom}(E,E''))$
that are ${\rm End}(E)$-morphisms are in bijection with
${\rm Hom}_{\rm Vect}(\underline{\mathbb{C}},{\rm Hom}(E',E''))\cong{\rm Hom}_{\rm Vect}(E',E'')$.
(This proof works quite generally in a symmetric monoidal
category with duals in which all objects have invertible dimension.)\\
Inspection shows that the compatibility with the morphisms $\varphi_{EE'}$ and $\varphi_{EE''}$ over $Y^{[2]}_{M}$ guarantees that $\tilde{f}$ is indeed a morphism of $\mathcal{G}$-modules.
\item Essential surjectivity: Let $\mathcal{B}\in{\rm pfmod-}\mathcal{A}^{\mathcal{M}}$. Recall that $\mathcal{A}^{\mathcal{M}}$  is a projective and finitely generated module over the algebra $C^{\infty}(M;\mathbb{C})$: indeed, $C^{\infty}(M;\mathbb{C})$ sits in the center of $\mathcal{A}^{\mathcal{M}}$, and the action is via multiplication. This implies that $\mathcal{B}$ is a finitely generated and projective $C^{\infty}(M;\mathbb{C})$-module as well. Hence, there exists an object $B \in {\rm Vect}(M)$ which is a right module over 
the algebra bundle $D(\tilde{\mathcal{F}}^{\mathcal{M}}(\mathcal{M}))$ on
$M$, and such that $\Gamma(B)\simeq{\mathcal{B}}$.\\
Consider the bundle $\pi^{*}B$ over $Y$. It is a finitely generated and projective right module over ${\rm End}(E)\simeq\pi^{*}D(\tilde{\mathcal{F}}^{\mathcal{M}}(\mathcal{M}))$. Notice that this means in particular that the fiber of $\pi^{*}B$ over a point $y$ is a finitely generated and projective module over ${\rm End}(E_{y})$. A classical result concerning endomorphism algebras states that for any finite dimensional vector space $V$ and any finitely generated and projective right module $Q$ over ${\rm End}(V)$, we have a canonical isomorphism
%%%
\begin{equation}
Q \simeq {\rm Hom}(V,Q\otimes_{\rm End(V)}V)
\end{equation} 
%%%
as right ${\rm End(V)}$-modules. Applying this result to vector bundles, we have
\begin{equation}
\pi^{*}B\simeq{\rm Hom}(E,\pi^{*}B\otimes_{{\rm End}(E)}E)
\end{equation}
Finally, we have to show that the bundle $\pi^{*}B\otimes_{{\rm End}(E)}E$ comes equipped with a gerbe module morphism. This is indeed induce by the morphism $\rho$ (recall that $\mathcal{M}=(E,\rho)$) as follows. First, recall that we have canonically
\begin{equation}\label{caniso}
{\rm End}(\pi_{1}^{*}E)\simeq{\rm End}(\pi_{1}^{*}E\otimes L)
\end{equation}
Moreover, the morphism $\rho:\pi_{1}^{*}E\otimes L\to\pi_{2}^{*}E$ is a module morphism along the algebra bundle morphism $\rho_{*}:{\rm End}(\pi_{1}^{*}E)\to{\rm End}(\pi^{*}_{2}E)$, where we have used the isomorphism (\ref{caniso}). We have then the following isomorphism
%%%
\begin{equation}
\begin{array}{rl}
\pi_{1}^{*}(\pi^{*}B\otimes_{{\rm End}(E)}E)\otimes L &\simeq (\pi_{1}^{*}\pi^{*}B\otimes_{{\rm End}(\pi_{1}^{*}E)}\pi_{1}^{*}E)\otimes L\\
  &\simeq\pi_{1}^{*}\pi^{*}B\otimes_{{\rm End}(\pi_{1}^{*}E)}(\pi_{1}^{*}E\otimes L)\\
 &\simeq\pi_{2}^{*}\pi^{*}B\otimes_{{\rm End}(\pi_{2}^{*}E)}\pi_{2}^{*}E\\
&\simeq\pi_{2}^{*}(\pi^{*}B\otimes_{{\rm End}(E)}E)
\end{array}
\end{equation}
%%%
which can proved to be compatible with $\mu$.
\end{enumerate}
\end{proof}
\section{Gerbe modules on Lie groupoids}\label{gerbegroupoids}
In this section we will see how Theorem \ref{main} can be extended from the category of smooth manifolds to the category of Lie groupoids.
%%%
\subsection{Vector bundles on Lie groupoids}\label{vbgroupoids}
%%%
In the following we recall some basic results about the category of vector bundles on Lie groupoids.\\
Let $G=(G_{0},G_{1})$ be a Lie groupoid with source and target map $s,t:G_{1}\to{G_{0}}$, respectively. Recall that a vector bundle over $G$ is given by a vector bundle $E\to{G_{0}}$ together with an isomorphism $\psi_{E}:s^{*}E\xrightarrow{\simeq}{t^{*}E}$ which is associative over ${G_{1}}\times_{G_{0}}G_{1}$. A morphism between two vector bundles $(E,\psi_{E})$ and $(F,\psi_{F})$ over $G$ is given by a morphism of vector bundles $f:E\to{F}$ over $G_{0}$ which is compatible with the isomorphisms $\psi_{E}$ and $\psi_{F}$ over ${G_{1}}$. Finite
rank vector bundles on $G$ and their morphisms form a $\mathbb{C}$-linear symmetric monoidal
category with duals ${\rm Vect(G)}$: indeed, the dual of $(E,\psi_{E})$ is given by $(E^{*},(\psi_{E}^{*})^{-1})$. The monoidal unit in ${\rm Vect(G)}$ is given by $(\underline{\mathbb{C}},id)$, which by abuse of notation we denote $\underline{\mathbb{C}}$. \\

Let $(\mathcal{C},\otimes)$ be a monoidal category, and let $U$ and $V$ in $\mathcal{C}$. The internal hom ${\rm \underline{Hom}}(U,V)$, if it exists, is an object in $\mathcal{C}$ for which there is an isomorphism
\begin{displaymath}
{\rm Hom}_{\mathcal{C}}(W,{\rm \underline{Hom}}(U,V)) \simeq {\rm Hom}_{\mathcal{C}}(W\otimes{U},V)
\end{displaymath} 
%%%
which is natural in $U$, $V$ and $W$.
\begin{lemma}
The category ${\rm Vect}(G)$ admits an internal hom ${\rm \underline{Hom}}(E,F)$ for all $E$ and $F$. 
\end{lemma}
%%%
\begin{proof}
Define 
\begin{equation}
{\rm \underline{Hom}}(E,F):=(E^{*}\otimes{F},(\psi_{E}^{*})^{-1}\otimes\psi_{F})
\end{equation}
and let $W\in {\rm Vect}(G)$. First, notice that
\begin{displaymath}
{\rm Hom}_{ {\rm Vect}(G)}(W,E^{*}\otimes{F})\subset {\rm Hom}_{ {\rm Vect}(G_{1})}(W,E^{*}\otimes{F})\simeq^{\Xi} {\rm Hom}_{ {\rm Vect}(G_{1})}(W\otimes {E},F)
\end{displaymath}
where 
\begin{equation}
\Xi(f):W\otimes{E}\xrightarrow{f\otimes{id}}E^{*}\otimes{F}\otimes{E}\xrightarrow{tr_{E}}{F}
\end{equation}
for $f\in{\rm Hom}(W,E^{*}\otimes{F})$ and $tr_{E}:E^{*}\otimes{E}\to\underline{\mathbb{C}}$ the canonical pairing.
%%%
By using the invariance of the trace pairing $tr_{E}$ under isomorphisms, we have the following commutative diagram
%%%
\begin{equation}
\xymatrixcolsep{5pc}\xymatrix{
s^{*}W\otimes s^{*}E \ar[r]^{\psi_{W}\otimes\psi_{E}} \ar[d]_{s^{*}(f)\otimes{\rm id}} &t^{*}W\otimes t^{*}E \ar[d]^{t^{*}(f)\otimes{\rm id}}\\
s^{*}E^{*}\otimes s^{*}F \otimes s^{*}E \ar[r]^{(\psi_{E}^{*})^{-1}\otimes\psi_{F}\otimes\psi_{E}} \ar[d]_{tr_{s^{*}E}} &t^{*}E^{*}\otimes t^{*}F \otimes t^{*}E \ar[d]^{tr_{t^{*}E}} \\
s^{*}F \ar[r]^{\psi_{F}} & t^{*}F
}
\end{equation}
%%%
for any $f\in{\rm Hom}_{ {\rm Vect}(G)}(W,E^{*}\otimes{F})$. The outer commutative square assures then that the isomorphism $\Xi$ maps ${\rm Hom}_{ {\rm Vect}(G)}(W,E^{*}\otimes{F})$ to 
the subspace
${\rm Hom}_{ {\rm Vect}(G)}(W\otimes{E},F) \subset {\rm Hom}_{ {\rm Vect}(G_1)}(W,E^{*}\otimes{F})$. Moreover, notice that the isomorphism $\Xi$ is natural in $W$, $E$, and $F$. 

\end{proof}
%%%
In particular, since the construction above is natural, we have the internal hom functor
\begin{displaymath}
\underline{\rm Hom}(-,-): {\rm Vect}(G)^{op} \times {\rm Vect}(G) \to {\rm Vect}(G)
\end{displaymath}
\begin{lemma}
For any $E\in {\rm Vect}(G)$, ${\rm \underline{End}}(E)$ is an algebra object
in ${\rm Vect}(G)$.
\end{lemma}
%%%
\begin{proof}
The multiplication morphism 
%%%
\begin{equation}
m:{\rm \underline{End}}(E) \otimes {\rm \underline{End}}(E) \to {\rm \underline{End}}(E)
\end{equation}
%%%
is induced by the vector bundle morphism
%%%
\begin{equation}
E^{*}\otimes E \otimes E^{*} \otimes E \xrightarrow{{\rm id}\otimes{c_{E^{*}E}}} E^{*}\otimes E \otimes E \otimes E^{*} \xrightarrow{{\rm id}\otimes{c_{EE}}\otimes{\rm id}} E^{*}\otimes E \otimes E \otimes E^{*} \xrightarrow{{\rm tr}_{E}\otimes{c_{EE^{*}}}} E^{*}\otimes E
\end{equation}
%%%
where $c_{AB}$ denote the symmetric braiding between the two vector bundles $A$ and $B$.\\
%%%
The associativity of $m$ is guaranteed by the associativity of the tensor product of vector bundles and bundle morphisms.
\end{proof}
%%%
Similarly, we have the following
%%%
\begin{lemma}
For any $E$ and $F$ in ${\rm Vect}(G)$, ${\rm \underline{Hom}}(E,F)$ is a right module over ${\rm \underline{End}}(E)$ and a left module over ${\rm \underline{End}}(F)$.
\end{lemma}
\begin{lemma}
For any $E$, $M$ and $N$ in ${\rm Vect}(G)$, we have a canonical 
isomorphism of vector spaces
\begin{displaymath}
{\rm Hom}_{\underline{\rm End}(E)}({\rm \underline{Hom}}(E,M),{\rm \underline{Hom}}(E,N))\simeq{\rm{Hom}_{{\rm Vect}(G)}(M,N)}
\end{displaymath}
\begin{proof}
The proof follows from 
\begin{displaymath}
\begin{array}{rl}
{\rm Hom}_{\underline{\rm End}(E)}({\rm \underline{Hom}}(E,M),{\rm \underline{Hom}}(E,N))&\simeq{\rm Hom}_{\underline{\rm End}(E)}(E^{*}\otimes M,E^{*} \otimes N))\\
&\simeq{\rm Hom}_{\underline{\rm End}(E)}(E^{*} \otimes M \otimes E, N)\\
&\simeq{\rm Hom}_{\underline{\rm End}(E)}(M \otimes {\underline {\rm End}}(E), N)\\
&\simeq{\rm Hom}_{{\rm Vect}(G)}(M, N)
\end{array}
\end{displaymath}
\end{proof}
%%%
\begin{remark} The results above hold in general for any symmetric monoidal category $\mathcal{C}$ with duals. See \cite{ostrik} for a discussion of these results in the case of more general tensor categories and module 
categories over them.
\end{remark}
%%%
\end{lemma}
We conclude this section on vector bundles over Lie groupoids with the following
%%%
\begin{lemma}
Let $L$ be a line bundle over a Lie groupoid $G=(G_{0},G_{1})$. Then there is a canonical isomorphism
\begin{equation}
L^{*} \otimes L \xrightarrow{\simeq} \underline{\mathbb{C}}
\end{equation}
\end{lemma}
%%%
\begin{proof}
Consider the isomorphism $c:L^{*}\otimes{L}\to\underline{\mathbb{C}}$ given by the canonical pairing. It is immediate to see that the following diagram commutes
\begin{displaymath}
\xymatrixcolsep{5pc}\xymatrix{
s^{*}L^{*}\otimes s^{*}L \ar[r]^{(\psi_{L}^{*})^{-1}\otimes\psi_{L}} \ar[d]_{s^{*}c} & t^{*}L^{*}\otimes t^{*}L\ar[d]^{t^{*}c} \\
\underline{\mathbb{C}} \ar[r]_{{\rm id}}& \underline{\mathbb{C}}
}
\end{displaymath}
\end{proof}
%%%
\subsection{Gerbes and gerbe modules on Lie groupoids}
%%%
Gerbes on Lie groupoids can be described in two equivalent ways: they can be seen as objects of a bicategory obtained via the plus contruction applied to the 2-stack of gerbes over manifolds \cite{niksch}, or by \emph{internalizing} to Lie groupoids the definition given in Section \ref{prelim}, where we substitute manifolds with Lie groupoids, line bundles over manifolds with line bundles over Lie groupoids, etc. \cite{husemoller}. In this paper, we will follow the second description, with the technical caveat to replace the notion of surjective submersion in manifolds with the appropriate one in Lie groupoids: as shown in \cite{niksch}, the appropriate notion is that of a \emph{weak equivalence} of groupoids. 
%%%
\begin{definition}
Let $\Gamma$ and $\Lambda$ be Lie groupoids. A morphism $F:\Gamma \to \Lambda$ is called a weak equivalence if
%%%
\begin{enumerate}
\item The diagram
\begin{displaymath}
\xymatrix{\Gamma_{1} \ar[r]^{F_{1}} \ar[d]_{s\times t} & \Lambda_{1}\ar[d]^{s\times t}\\
\Gamma_{0}\times\Gamma_{0} \ar[r]_{F_{0}\times{F_{0}}} & \Lambda_{0}\times\Lambda_{0}
}
\end{displaymath}
is a pullback diagram.
\item The smooth map 
\begin{displaymath}
\Gamma_{0}\times_{\Lambda_{0}}\Lambda_{1} \to \Lambda_{0}
\end{displaymath}
is a surjective submersion.
\end{enumerate}
%%%
\end{definition}
%%%
To generalise the notion of a bundle gerbe over a manifold, we need the notion of a fiber product of groupoids.
%%%
\begin{definition}
Let $\Gamma^{1}$, $\Gamma^{2}$ and $\Lambda$ be Lie groupoids, and let $F^{1}:\Gamma^{1} \to \Lambda$ and $F^{2}:\Gamma^{2} \to \Lambda$ be morphisms of Lie groupoids. The fiber product $\Gamma^{1}\times_{\Lambda}\Gamma^{2}$ is the groupoid such that
%%%
\begin{displaymath}
{\rm Obj}(\Gamma^{1}\times_{\Lambda}\Gamma^{2}):=\left\{(a,b,\alpha) \in \Gamma^{1}_{0}\times \Gamma^{2}_{0}\times\Lambda_{1}: s(\alpha)=F_{0}^{1}(a), t(\alpha)=F_{0}^{2}(b)\right\}
\end{displaymath}
and
%%%
\begin{displaymath}
{\rm Mor}((a,b,\alpha),(a',b',\alpha')):=\left\{(f,g)\in {\rm Mor}(a,a')\times{\rm Mor}(b,b'):\alpha\circ F^2_{1}(g)=F^1_{1}(f)\circ\alpha'\right\}
\end{displaymath}
%%%
\end{definition}
%%%
One can show that if $F^{1}$ or $F^{2}$ is a weak equivalence of groupoids, then $\Gamma^{1}\times_{\Lambda}\Gamma^{2}$ is also a Lie groupoid
(see \cite{moer}, section 2.3).\\
Given a weak equivalence $\Gamma \to \Lambda$,  we will use the notation $\Gamma^{[k]}$ for the $k$-th fibered product of $\Gamma$ over $\Lambda$.\\
A bundle gerbe over a Lie groupoid $\Lambda$ is then a triple $(\Gamma, L, \mu)$, where $\Gamma \xrightarrow{\pi} \Lambda$ is a weak equivalence, $L$ a line bundle over $\Gamma^{[2]}$, and $\mu$ an isomorphism over $\Gamma^{[3]}$ satisfying an associativity condition over $\Gamma^{[4]}$.
Similarly, Definition \ref{gmodule} carries through to Lie groupoids immediately, and we obtain in particular a category of bundle gerbe modules. We have a descent category $Desc(\Gamma \to \Lambda)$ of vector bundles associated to a weak equivalence $\Gamma \to \Lambda$ of groupoids, defined by internalizing the construction in Section \ref{descent}. Moreover, we have an equivalence ${\rm Vect}(\Lambda)\simeq Desc(\Gamma \to \Lambda) $ induced by the pullback functor \cite{niksch}.\\

Let $\mathcal{G}=(\Gamma,L,\mu)$ be a bundle gerbe over $\Lambda$, and let $\mathcal{M}=(E,\rho)$ and $\mathcal{N}=(E',\rho')$ $\mathcal{G}$-modules. Consider the homomorphism
%%%
\begin{equation}
 \varphi_{EE'}:\pi_{2}^{*}{\underline{\rm Hom}}(E,E')\to\pi_{1}^{*}{\underline{\rm Hom}}(E,E')
\end{equation}
%%%
defined as in (\ref{desciso}).\\
Since all the properties of the category of vector bundles over manifolds used there extend directly to the category of vector bundles over Lie groupoids, we have, as shown in Section \ref{vbgroupoids},
\begin{lemma}
$(\underline{{\rm Hom}}(E,E'),\varphi_{EE'}^{-1})$ is an object in $Desc(\Gamma\to{\Lambda})$.
\end{lemma}
In the same spirit we have the following
\begin{lemma}Let $\mathcal{M}=(E,\rho)$, $\mathcal{N}=(E{'},\rho{'})$ and $\mathcal{P}=(E{''},\rho{''})$ be $\mathcal{G}$-modules, and let $f\in{\rm Hom}_{\mathcal{G}}(\mathcal{M},\mathcal{N})$. Then $f$ induces a morphism $(\underline{{\rm Hom}}(E,E{'}),\varphi_{EE{'}}^{-1})\to(\underline{{\rm Hom}}(E,E{''}),\varphi_{EE{''}}^{-1})$ in $Desc(\Gamma\to{\Lambda})$.
\end{lemma}
As in Section \ref{descent}, given a nontrivial $\mathcal{G}$-module $\mathcal{M}$ we can set up a functor
\begin{equation}
\tilde{\mathcal{F}}_{Grp}^{\mathcal{M}}:\mathcal{G}-{\rm mod}\to Desc(\Gamma\to{\Lambda})
\end{equation}
Since the internal Hom functor of vector bundles on Lie groupoids is $\mathbb{C}$-linear and faithful, the functor $\tilde{\mathcal{F}}^{\mathcal{M}}$ inherits these properties, hence we have
\begin{proposition}
The functor $\tilde{\mathcal{F}}^{\mathcal{M}}$ is a faithful $\mathbb{C}$-linear functor.
\end{proposition}
\subsection{Modules of sections and Serre-Swan Theorem}
Recall that an \'etale Lie groupoid is a Lie groupoid for which the
source map (and, as a consequence, all the structure maps -- target, 
identity and composition) is a local diffeomorphism.
Let $G=(G_{0},G_{1})$ be an \'etale Lie groupoid, and let $C_{c}^{\infty}(G)$ be its convolution algebra. In particular, $C_{c}^{\infty}(G)$ admits the structure of a \emph{Hopf algebroid} over the commutative algebra $C^{\infty}_{c}(G_{0})$. Briefly, a Hopf algebroid is given by an algebra $A$ together with a commutative subalgebra $A_{0}$ in which $A$ has local units, and is equipped with a commutative coalgebra stucture $(\Delta, \epsilon)$ over the right $A_{0}$-action, and a linear antipode $S$ satisfying a certain number of axioms. See \cite{mrcun}, Definition 2.1 for details. In particular, in the case of the Hopf algebroid associated to the groupoid $G$, we have 
\begin{enumerate}
\item the algebra $A$ is given by $C_{c}^{\infty}(G)$
\item the commutative algebra $A_{0}$ is given by $C^{\infty}_{c}(G_{0})$
\item the counit $\epsilon: C_{c}^{\infty}(G) \to C^{\infty}_{c}(G_{0})$ is given by
\begin{displaymath}
\epsilon(a)(x):=\sum_{s(g)=x}a(g)
\end{displaymath}
for any $a\in C_{c}^{\infty}(G)$ and $x\in {G_{0}}$. (This expression 
makes sense, since the Lie groupoid is \'etale.)
\item the antipode $S: C_{c}^{\infty}(G) \to C_{c}^{\infty}(G)$ is given by
%%%
\begin{displaymath}
S(a)(g):=a(g^{-1})
\end{displaymath}
%%%
for any $a\in C_{c}^{\infty}(G)$ and $g\in {G}$.
%%%
\item the comultiplication $\Delta:C_{c}^{\infty}(G) \to C_{c}^{\infty}(G)\otimes_{C_{c}^{\infty}(G_{0})}C_{c}^{\infty}(G)$ is given by the composition
\begin{displaymath}
 C_{c}^{\infty}(G) \to C_{c}^{\infty}(G\times_{s}G) \to C_{c}^{\infty}(G)\otimes_{C_{c}^{\infty}(G_{0})}C_{c}^{\infty}(G)
\end{displaymath}
%%%
where the first homomorphism is induced by the diagonal embedding of $G$ in $G\times_{s}G$, and the second is given by the inverse of the isomorphism $\Omega:C_{c}^{\infty}(G)\otimes_{C_{c}^{\infty}(G_{0})}C_{c}^{\infty}(G) \to C_{c}^{\infty}(G\times_{s}G)$ 
given by
%%%
\begin{displaymath}
\Omega(a\otimes{a'})(g,g'):=a(g)a'(g')
\end{displaymath}
%%%
\end{enumerate} 

Let $E$ be a vector bundle over the groupoid $G$, and denote with $\Gamma_{c}(E)$ the vector space of smooth sections of the vector bundle $E\to{G_{0}}$. As shown in \cite{kalisnik}, $\Gamma_{c}(E)$ admits a left action of the Hopf algebroid $C_{c}^{\infty}(G)$ associated to $G$. Indeed, we have a bilinear map
%%%
\begin{equation}
C_{c}^{\infty}(G)\times\Gamma_{c}(E) \to \Gamma_{c}(E)
\end{equation}
%%%
given by
%%%
\begin{equation}
(af)(x):=\sum_{t(g)=x}a(g)(g\cdot f(s(g)))
\end{equation}
%%%
Given a morphism $\varphi:E \to F$ of vector bundles over $G$, the induced homomorphism of $C_{c}^{\infty}(G_{0})$-modules $\Gamma(\varphi):\Gamma_{c}(E) \to \Gamma_{c}(F)$ is also a homomorphism of left $C_{c}^{\infty}(G)$-modules. Hence, we have a functor 
%%%
\begin{equation}
\Gamma_{c}: {\rm Vect}(G) \to {_{G}{\rm Mod}}
\end{equation}
%%%
from the category of vector bundles over the groupoid $G$ to the category of left modules over the Hopf algebroid $C_{c}^{\infty}(G)$ of the groupoid $G$. 
%%%
\begin{definition}
Let $G=(G_{0}, G_{1})$ be an \'etale Lie groupoid. A left module $M$ over $C_{c}^{\infty}(G)$ is said to be \emph{of finite type} if it is isomorphic as a $C_{c}^{\infty}(G_{0})$-module to some submodule of the module  $C_{c}^{\infty}(G_{0})^{k}$ for some natural number $k$.
\end{definition}
%%%
Given a left module $M$ over $C_{c}^{\infty}(G)$ of finite type, consider for any $x\in{G_{0}}$ the $C_{c}^{\infty}(G_{0})$-module $I_{x}M:=I_{x}C_{c}^{\infty}(G_{0})\cdot{M}$, where $I_{x}C_{c}^{\infty}(G_{0}):=\left\{f \in C_{c}^{\infty}(G_{0}): f(x)=0 \right\}$. Denote
%%%
\begin{equation}
M(x):=M/I_{x}M
\end{equation}
%%%
the quotient $C_{c}^{\infty}(G_{0})(x)$-module, where $C_{c}^{\infty}(G_{0})(x):=C_{c}^{\infty}(G_{0})/I_{x}C_{c}^{\infty}(G_{0})$. Since $M$ is of finite type, one can prove that ${\rm dim}_{\mathbb{C}}M(x) < \infty$.
%%%
\begin{definition}
Let $G=(G_{0}, G_{1})$ be an \'etale Lie groupoid. A left module $M$ over $C_{c}^{\infty}(G)$ of finite type is said to be \emph{of constant rank} if the function
%%%
\begin{displaymath}
x \mapsto {\rm dim}_{\mathbb{C}}M(x)
\end{displaymath}
%%%
is constant over $G_{0}$.
%%%
\end{definition}
%%%
%%%
Let $M$ and $N$ be modules of finite type and constant rank over $C_{c}^{\infty}(G)$. Their direct sum $M\oplus{N}$ is again of finite type and of constant rank. Moreover, the tensor product $M\otimes_{C_{c}^{\infty}(G_{0})}N$ can be given the structure of a left $C_{c}^{\infty}(G)$-module\cite{kalisnik}. With this tensor product we have the following
%%%
\begin{lemma}
Let $G=(G_{0}, G_{1})$ be an \'etale groupoid. Modules of finite type and of constant rank over the Hopf algebroid $C_{c}^{\infty}(G)$ and module morphisms form a monoidal tensor category ${\rm Mod}(G)$.
\end{lemma}
%%%
Notice that the unit in ${\rm Mod}(G)$ is given by the algebra $C_{c}^{\infty}(G_{0})$ equipped with its natural left $C_{c}^{\infty}(G)$-action.
%%%
Moreover, we have the following 
\begin{lemma}\label{sections}
Let $G=(G_{0}, G_{1})$ be an \'etale groupoid, and let $E$ be a vector bundle over $G$. Then the left $C_{c}^{\infty}(G)$-module $\Gamma_{c}(E)$ is a module of finite type and of constant rank.
\end{lemma}
%%%%
Lemma \ref{sections} tells us that the section functor $\Gamma_{c}$ induces a tensor functor
\begin{equation}
\Gamma_{c}: {\rm Vect}(G) \to {\rm Mod}(G)
\end{equation}
%%%%
from the category of vector bundles over the groupoid $G$ to the category of modules of constant type and of finite rank over the Hopf algebroid $C_{c}^{\infty}(G)$ of $G$.\\
%%%%
Finally, we have a Serre-Swan type theorem for vector bundles over \'etale Lie groupoids\cite{kalisnik}.
%%%
\begin{theorem} \label{4.3.5}The functor $\Gamma_{c}: {\rm Vect}(G) \to {\rm Mod}(G)$ is an equivalence of tensor categories for any \'etale Lie groupoid.
\end{theorem}
%%%%
\subsection{Equivalence of categories}
%%%%
Let $G=(G_{0},G_{1})$ be an \'etale groupoid, and we will assume that the manifold of objects $G_{0}$ is compact. Let $\mathcal{G}=(\Gamma, L, \mu)$ be a bundle gerbe over $G$, and $\mathcal{M}=(E,\rho)$ a nontrivial $\mathcal{G}$-module. Similarly to Proposition \ref{azumend}, we can assign to $\mathcal{M}$ an infinite dimensional algebra defined as
%%%%
\begin{equation}
\mathcal{A}^{\mathcal{M}}_{Grp}:=\Gamma(D(\tilde{\mathcal{F}}_{Grp}^{\mathcal{M}}(\mathcal{M})))
\end{equation}
%%%%
$\mathcal{A}^{\mathcal{M}}_{Grp}$ is an algebra object with unit in ${\rm Mod}(G)$, with the unit morphism $C^{\infty}(G_{0})\to\mathcal{A}^{\mathcal{M}}_{Grp}$ given by the embedding via the identity section.\\
%%%%
Moreover, for any other $\mathcal{G}$-module $\mathcal{N}$ we obtain a right $\mathcal{A}^{\mathcal{M}}_{Grp}$-module given by
%%%%
\begin{equation}
\mathcal{A}^{\mathcal{M}\mathcal{N}}_{Grp}:=\Gamma(D(\tilde{\mathcal{F}}_{Grp}^{\mathcal{M}}(\mathcal{N})))
\end{equation}
%%%%
Consider the category ${\rm Mod}-\mathcal{A}^{\mathcal{M}}_{Grp}$ of right $\mathcal{A}^{\mathcal{M}}_{Grp}$-modules in ${\rm Mod}(G)$, the category of modules of constant type and of finite rank over the Hopf algebroid $C_{c}^{\infty}(G)$ of $G$.

Using the same arguments as in Section \ref{equivcat}, we have a functor
%%%
\begin{equation}
\mathcal{F}_{Grp}^{\mathcal{M}}: \mathcal{G}-{\rm mod} \to {\rm Mod}-\mathcal{A}^{\mathcal{M}}_{Grp}
\end{equation}
%%%
By using the results in Section \ref{gerbegroupoids}, the proof of Theorem \ref{main} extends immediately to the proof of
%%%
\begin{theorem} Let $G=(G_{0},G_{1})$ be an \'etale groupoid, with $G_{0}$ a compact manifold. Let $\mathcal{G}=(\Gamma,L,\mu)$ be a bundle gerbe over $G$ admitting a nontrivial $\mathcal{G}$-module $\mathcal{M}$. Then the functor $\mathcal{F}_{Grp}^{\mathcal{M}}$ is $\mathbb{C}$-linear, fully faithful and essentially surjective, hence an equivalence of categories.  
\end{theorem}
\quad\\
\noindent{\sc Aknowledgments}: We thank Domenico Fiorenza, Thomas Nikolaus and Urs Schreiber for helpful discussions and correspondence. CS and AV are partially supported by the Collaborative Research Centre 676 
``Particles, Strings and the Early Universe - the Structure of 
Matter and Space-Time'' and by the DFG Priority Programme 1388 
``Representation Theory''.
%%%

\end{document}